\documentclass[11pt]{article}
\usepackage{graphicx,indentfirst}
\usepackage{amscd}

\usepackage[all]{xy}
\usepackage{amsmath}
\usepackage{amsfonts,amssymb}
\usepackage{extarrows}
\usepackage{mathscinet}
\usepackage{amsthm}
\usepackage{latexsym}
\usepackage[american]{babel}
\usepackage{times}
\usepackage{tikz}
\usepackage{tikz-cd}
\usepackage{url}
\usepackage{comment}
\usepackage[affil-it]{authblk}

\usetikzlibrary{matrix,arrows,decorations.pathmorphing}
\usepackage[a4paper,left=2.5cm,right=2.5cm,bottom=3cm,top=3.5cm]{geometry}

\makeatletter
\newcommand*{\relrelbarsep}{.386ex}
\newcommand*{\relrelbar}{%
  \mathrel{%
    \mathpalette\@relrelbar\relrelbarsep
  }%
}
\newcommand*{\@relrelbar}[2]{%
  \raise#2\hbox to 0pt{$\m@th#1\relbar$\hss}%
  \lower#2\hbox{$\m@th#1\relbar$}%
}
\providecommand*{\rightrightarrowsfill@}{%
  \arrowfill@\relrelbar\relrelbar\rightrightarrows
}
\providecommand*{\leftleftarrowsfill@}{%
  \arrowfill@\leftleftarrows\relrelbar\relrelbar
}
\providecommand*{\xrightrightarrows}[2][]{%
  \ext@arrow 0359\rightrightarrowsfill@{#1}{#2}%
}
\providecommand*{\xleftleftarrows}[2][]{%
  \ext@arrow 3095\leftleftarrowsfill@{#1}{#2}%
}
\makeatother

\RequirePackage[T1]{fontenc}

\begin{document}

\title{Descent of dg cohesive modules for open covers on complex manifolds}

\author{Zhaoting Wei
 \thanks{Email: \texttt{zhaoting.wei@tamuc.edu}}}
\affil{Department of Mathematics, Texas A\&M University-Commerce\\Commerce, TX, 75429, USA}

\maketitle

\newtheorem{thm}{Theorem}[section]
\newtheorem{lemma}[thm]{Lemma}
\newtheorem{prop}[thm]{Proposition}
\newtheorem{coro}[thm]{Corollary}
\newtheorem{ques}[thm]{Question}
\newtheorem{conj}[thm]{Conjecture}
\theoremstyle{definition}\newtheorem{defi}{Definition}[section]
\theoremstyle{remark}\newtheorem{eg}{Example}
\theoremstyle{remark}\newtheorem{rmk}{Remark}
\theoremstyle{remark}\newtheorem{ctn}{Caution}

\newcommand{\Holim}{\text{Holim}}
\newcommand{\diff}{\text{d}}
\newcommand{\For}{\text{For}}
\newcommand{\Mod}{\text{Mod}}
\newcommand{\Modb}{\text{Mod-}}
\newcommand{\perf}{\text{perf}}
\newcommand{\Perf}{\text{Perf}}
\newcommand{\coh}{\text{coh}}
\newcommand{\id}{\text{id}}

\begin{abstract}
In this paper we study the descent problem of cohesive modules on  complex manifolds. For a complex manifold $X$ we could consider the Dolbeault dg-algebra $\mathcal{A}(X)$ on it and Block in 2006 introduced a dg-category $\mathcal{P}_{\mathcal{A}(X)}$, called cohesive modules, associated with $\mathcal{A}(X)$. The same construction works for any open subset $U\subset X$ and we obtain a dg-presheaf on $X$ given by $U\mapsto \mathcal{P}_{\mathcal{A}(U)}$. In this paper we prove that this dg-presheaf satisfies the descent property for any locally finite open cover of a complex manifold $X$. This generalizes part of the  results of Ben-Bassat and Block in 2012, which studied the case that $X$ is covered by two open subsets.


\end{abstract}

\section{Introduction}
In \cite{block2010duality}, Block assigned a dg-category $\mathcal{P}_{\mathcal{A}}$, called \emph{cohesive modules},  to a (curved) dg-algebra $\mathcal{A}$. The dg-category of cohesive modules provides a way to enhance many well-known triangulated categories and has been studied in \cite{block2010mukai}, \cite{block2014higher}, \cite{yu2016dolbeault}, \cite{qiang2016bott}.

In particular, for a compact complex manifold $X$ we consider the Dolbeault dg algebra $\mathcal{A}^{\bullet}(X)=(A^{0,\bullet}(X),\bar{\partial})$. In this case a cohesive module consists of a complex of smooth vector bundles on $X$  with a $\bar{\partial}$-$\mathbb{Z}$-connection. Block proved in \cite{block2010duality} that $\mathcal{P}_{\mathcal{A}(X)}$ gives a dg-enhancement of $D^b_{\text{coh}}(X)$, the bounded derived category of coherent sheaves on $X$. Based on this result, in \cite{bismut2021coherent} the authors proved the Grothendieck-Riemann-Roch theorem for coherent sheaves on compact complex manifolds. Moreover, \cite{chuang2021maurer} generalizes the result in \cite{block2010duality} to the case that $X$ is non-compact with a slightly more restricted definition of coherent sheaves.

The descent problem is also one of the original motivations of considering dg-enhancement of triangulated categories. In \cite{ben2013milnor} Ben-Bassat and Block proved that for a compact complex manifold $X$ and an open cover $\{U_1,U_2\}$ of $X$, the natural restriction functor
\begin{equation}\label{equation: descent for two open subsets}
\mathcal{P}_{\mathcal{A}(X)}\to \mathcal{P}_{\mathcal{A}(\overline{U_1})}\times^h_{\mathcal{P}_{\mathcal{A}(\overline{U_1}\cap \overline{U_2})}} \mathcal{P}_{\mathcal{A}(\overline{U_2})}
\end{equation}
is a quasi-equivalence of dg-categories, where $\mathcal{P}_{\mathcal{A}(\overline{U_1})}\times^h_{\mathcal{P}_{\mathcal{A}(\overline{U_1}\cap \overline{U_2})}} \mathcal{P}_{\mathcal{A}(\overline{U_2})}$ denotes the homotopy fiber product. See \cite{ben2013milnor} Theorem 6.7 and Theorem 7.4 for details.

\begin{rmk}
It is well-known that for derived categories, the natural restriction functor $D^b_{\text{coh}}(X)\to D^b_{\text{coh}}(\overline{U_1})\times^h_{D^b_{\text{coh}}(\overline{U_1}\cap \overline{U_2}) } D^b_{\text{coh}}(\overline{U_2})$ is not an equivalence of triangulated categories, see \cite{bertrand2011lectures} Section 2.2.
\end{rmk}

In this paper we study the descent of $\mathcal{P}_{\mathcal{A}(X)}$ for an arbitrary locally finite open cover $\mathcal{U}=\{U_i\}$ of $X$. In this case the homotopy fiber product on the right hand side of (\ref{equation: descent for two open subsets}) should be replaced by the homotopy limit, and we prove that the natural  functor
\begin{equation}\label{equation: descent for finite open subsets, introduction}
\mathcal{P}_{\mathcal{A}(X)}\to \Holim_{\mathcal{U}} \mathcal{P}_{\mathcal{A}(U_i)}
\end{equation}
is a quasi-equivalence of dg-categories.

According to \cite{block2017explicit}, the homotopy limit $\Holim_{\mathcal{U}} \mathcal{P}_{\mathcal{A}(U_i)}$ is quasi-equivalent to the dg-category of \emph{twisted complexes}, $\text{Tw}(X,\mathcal{P}_{\mathcal{A}},U_i)$ (See Section \ref{section: review of twisted complex} below). Therefore the main result of this paper could be stated as:

\begin{thm}\label{thm: cohesive modules satisfies descent for open covers, introduction}[See Theorem \ref{thm: cohesive modules satisfies descent for open covers} below]
Let $X$ be a  complex manifold and $\{U_i\}$ be a  finite open cover of $X$. Let $\mathcal{A}=(\mathcal{A}^{0,\bullet},\bar{\partial},0)$ be the Dolbeault dg-algebra on $X$ and $\mathcal{P}_{\mathcal{A}}$ be the dg-category of cohesive modules. Let $\text{Tw}(X, \mathcal{P}_{\mathcal{A}}, U_i)$ be the dg-category of globally bounded twisted complexes on $X$. Then the natural functor
$$
\mathcal{T}:\mathcal{P}_{\mathcal{A}(X)}\to \text{Tw}(X, \mathcal{P}_{\mathcal{A}}, U_i)
$$
is a dg-quasi-equivalence of dg-categories.
\end{thm}

This result can be considered as a dg-categorification of the \v{C}ech-Dolbeault theory of cohomologies on complex manifolds.

Let us briefly mention the strategy of the proof. We want to construct a right adjoint functor of $\mathcal{T}$, $\mathcal{S}:\text{Tw}(X, \mathcal{P}_{\mathcal{A}}, U_i)\to \mathcal{P}_{\mathcal{A}(X)}$. However, we will see that the image of $\mathcal{S}$ cannot be contained in $\mathcal{P}_{\mathcal{A}(X)}$. Therefore we have to enlarge our dg-category to quasi-cohesive modules $\mathcal{C}_{\mathcal{A}(X)}$ and get an adjoint pair
$$
\mathcal{T}:\mathcal{C}_{\mathcal{A}(X)} \rightleftarrows \text{Tw}(X, \mathcal{C}_{\mathcal{A}}, U_i): \mathcal{S}.
$$
Then the proof of Theorem \ref{thm: cohesive modules satisfies descent for open covers, introduction} consists of $(1)$ a detailed study of the pair $\mathcal{T}\rightleftarrows \mathcal{S}$ restricted to underlying complexes; and $(2)$ some general results on dg-categories and cohesive modules.

\begin{rmk}
The strategy of the proof of Theorem \ref{thm: cohesive modules satisfies descent for open covers, introduction} is similar to the proof of the main theorem in \cite{ben2013milnor}. Nevertheless the way of gluing underlying complexes in this paper is very different from that in \cite[Section 5]{ben2013milnor}. See Section \ref{section: the gluing of underlying complexes} below.
\end{rmk}

\begin{rmk}
Actually we will prove Theorem \ref{thm: cohesive modules satisfies descent for open covers, introduction} for locally finite open cover and the dg-category of \emph{globally bounded} twisted complexes. See Section \ref{subsec: globally bounded twisted} and Theorem \ref{thm: cohesive modules satisfies descent for open covers} below. If the cover is finite, then any twisted complex is globally bounded.
\end{rmk}

This paper is organized as follows: In Section \ref{section: review of cohesive modules} we review cohesive modules; in Section \ref{section: review of twisted complex} we review twisted complexes. In particular we define the natural functor $
\mathcal{T}: \mathcal{P}_{\mathcal{A}(X)}\to \text{Tw}(X, \mathcal{P}_{\mathcal{A}}, U_i)
$ as well as its right adjoint $\mathcal{S}$. In Section \ref{section: the gluing of underlying complexes} we temporarily ignore the $\bar{\partial}$-$\mathbb{Z}$-connection and focus on the gluing of underlying complexes. In Section \ref{section: the descent of cohesive modules} we take the $\bar{\partial}$-$\mathbb{Z}$-connection back and consider the descent of cohesive modules, where the main result of this paper is proved (Theorem \ref{thm: cohesive modules satisfies descent for open covers}). 

\section*{Acknowledgment}
The author wants to thank Jonathan Block for numerous discussions. The author also wants to thank Valery Lunts, Olaf Schn\"{u}rer, and Julian Holstein for helpful comments. Part of this paper was finished when the author was visiting the Institut des Hautes \'{E}tudes Scientifiques France in January 2018, and he wants to thank IHES for its hospitality and its ideal working environment.

\section{A review of cohesive modules}\label{section: review of cohesive modules}
\subsection{Definition and basic facts}
Let us first recall the definition of the cohesive module in \cite{block2010duality}
\begin{defi}\label{definition of cohesive module}[\cite{block2010duality} Definition 2.4]
For a curved dg-algebra $\mathcal{A} = (\mathcal{A}^{\bullet}, \diff_{\mathcal{A}}, c)$, we define the dg-category $\mathcal{P}_{\mathcal{A}}$:
\begin{enumerate}
\item
An object $E=(E^{\bullet}, \mathbb{E})$ in $\mathcal{P}_{\mathcal{A}}$, which we call a \emph{cohesive module}, is a $\mathbb{Z}$-graded (but bounded in both directions) right module $E^{\bullet}$ over $\mathcal{A}^0$, ($\mathcal{A}^0$ is the zero degree part of $\mathcal{A}^{\bullet}$) which is finitely generated and projective, together with a $\mathbb{Z}$-connection, which satisfies the usual Leibniz condition
$$
\mathbb{E}(e\cdot \omega)= \mathbb{E}(e)\cdot \omega+(-1)^{|e|}e\cdot \diff_{\mathcal{A}}(\omega)
$$
   $$
    \mathbb{E}: E^{\bullet}\otimes_{\mathcal{A}^0} \mathcal{A}^{\bullet} \rightarrow E^{\bullet}\otimes_{\mathcal{A}^0} \mathcal{A}^{\bullet}
    $$
    that satisfies the integrability condition that the relative curvature vanishes
    $$
    F_E(e)= \mathbb{E}\circ\mathbb{E}(e)+e\cdot c=0
    $$
    for all $e\in E^{\bullet}$.
\item The morphisms of degree $k$, $\mathcal{P}^k_{\mathcal{A}}(E_1,E_2)$ between two cohesive modules $E_1=(E_1^{\bullet}, \mathbb{E}_1)$ and $E_2=(E_2^{\bullet}, \mathbb{E}_2)$ are
    $$
    \{\phi: E^{\bullet}_1\otimes_{\mathcal{A}^0} \mathcal{A}^{\bullet} \rightarrow E^{\bullet}_2\otimes_{\mathcal{A}^0} \mathcal{A}^{\bullet}\text{ of degree } k \text{ and } \phi(ea)=\phi(e)a,~ \forall ~a\in \mathcal{A}^{\bullet} \}
    $$
    with differentials defined by
    $$
    \diff(\phi)(e)=\mathbb{E}_2(\phi(e))-(-1)^{|\phi|}
    $$
\end{enumerate}
\end{defi}

The differential $\mathbb{E}$ decomposes into $
\mathbb{E}=\sum \mathbb{E}^i$
where
$$
\mathbb{E}^i: E^{\bullet}\rightarrow E^{\bullet+1-i}\otimes_{\mathcal{A}}\mathcal{A}^i.
$$
A similar decomposition applies to the morphism $\phi$.

As  in \cite{block2010duality} Section 2.4 we could define the shift functor and the mapping cone in $\mathcal{P}_{\mathcal{A}}$.
\begin{defi}\label{defi: shift and mapping cone in PA}
For $E=(E^{\bullet},\mathbb{E})$, set $E[1]=(E[1]^{\bullet},\mathbb{E}[1])$, where $E[1]^{\bullet}=E^{\bullet+1}$ and $\mathbb{E}[1])=-\mathbb{E}$. Next for $\phi: E\to F$ a degree zero closed morphism in $\mathcal{P}_{\mathcal{A}}$, we define the mapping cone of $\phi$, $\text{Cone}(\phi)=(\text{Cone}(\phi)^{\bullet},\mathbb{C}_{\phi})$ by
$$
\text{Cone}(\phi)^{\bullet}=\begin{pmatrix}
  F^{\bullet} \\
  \oplus\\
  E[1]^{\bullet}
 \end{pmatrix}
$$
and
$$
\mathbb{C}_{\phi}=\begin{pmatrix}
  \mathbb{F}& \phi \\

0&   \mathbb{E}[1]
 \end{pmatrix}.
$$
\end{defi}

It is clear that $\mathcal{P}_{\mathcal{A}}$ is a pre-triangulated dg-category hence its homotopy category $\text{Ho}(\mathcal{P}_{\mathcal{A}})$ is a triangulated category.

In this paper we also consider the degree zero part $\mathcal{A}^0$ as a dg-algebra concentrated at degree zero with zero differential. In this case $\mathcal{P}_{\mathcal{A}^0}$ is simply the dg-category of bounded complexes of finitely generated projective  $\mathcal{A}^0$-modules. Let $\For: \mathcal{P}_{\mathcal{A}}\to \mathcal{P}_{\mathcal{A}^0}$ be the forgetful functor. It is clear that
$$
\For(E^{\bullet},\mathbb{E})=(E^{\bullet},\mathbb{E}^0).
$$

It is also useful to consider the larger dg-category of quasi-cohesive modules $\mathcal{C}_{\mathcal{A}}$: a quasi-cohesive module is a data $(\mathcal{Q}^{\bullet},\mathbb{Q})$ where everything is the same as in a cohesive module except that $\mathcal{X}^{\bullet}$ is not required to be bounded, finitely generated and projective. Similarly we have $\mathcal{C}_{\mathcal{A}^0}$.

We have the following definition of quasi-equivalence between quasi-cohesive modules.

\begin{defi}\label{defi: quasi-isomorphism between quasi-cohesive modules}
A degree zero closed morphism $\phi\in \mathcal{C}_{\mathcal{A}}(E_1,E_2)$ is called a \emph{quasi-isomorphism} if and only if $\For(\phi):\For(E_1)\to \For(E_2)$ is a quasi-isomorphism of complexes of $\mathcal{A}^0$-modules.
\end{defi}

The following result characterize homotopy equivalences in $\mathcal{P}_{\mathcal{A}}$.

\begin{prop}\label{prop: homotopy equivalences in PA}[\cite{block2010duality} Proposition 2.9]
Let $E_1$ and $E_2$ be objects in $\mathcal{P}_{\mathcal{A}}$. Then a degree zero closed morphism $\phi\in \mathcal{P}_{\mathcal{A}}(E_1,E_2)$ is a homotopy equivalence if and only if $\phi$ is a quasi-isomorphism in the sense of Definition \ref{defi: quasi-isomorphism between quasi-cohesive modules}, i.e. $\For(\phi):\For(E_1)\to \For(E_2)$ is a quasi-isomorphism of complexes of $\mathcal{A}^0$-modules.
\end{prop}
\begin{proof}
See \cite{block2010duality} Proposition 2.9.
\end{proof}

\begin{rmk}
The result in Proposition \ref{prop: homotopy equivalences in PA} is not true if one of $E_1$ and $E_2$ is not in $\mathcal{P}_{\mathcal{A}}$.
\end{rmk}

For $\mathcal{P}_{\mathcal{A}}$ the fully faithful Yoneda embedding $h: Z^0(\mathcal{P}_{\mathcal{A}})\to \Modb\mathcal{P}_{\mathcal{A}}$ is given by
$$
E\mapsto h_E=\mathcal{P}_{\mathcal{A}}(\cdot,E).
$$
Similarly for $\mathcal{C}_{\mathcal{A}}$ we have a fully faithful functor $\tilde{h}: Z^0(\mathcal{C}_{\mathcal{A}})\to \Modb\mathcal{P}_{\mathcal{A}}$ given by
$$
Q\mapsto \tilde{h}_Q=\mathcal{C}_{\mathcal{A}}(\cdot,Q)
$$
considered as a module over $\mathcal{P}_{\mathcal{A}}$.

We recall the following result, which justifies the name "quasi-isomorphism" in Definition \ref{defi: quasi-isomorphism between quasi-cohesive modules}.

\begin{prop}\label{prop: quasi-isomorphism of quasi-cohesive modules induces quasi-isomorphism for Yoneda embedding}[\cite{block2010duality} Proposition 3.9]
Let $\phi: Q_1\to Q_2$ be a quasi-isomorphism in $\mathcal{C}_{\mathcal{A}}$ as in Definition \ref{defi: quasi-isomorphism between quasi-cohesive modules}. Then the induced morphism $
\tilde{h}_{\phi}: \tilde{h}_{Q_1}\to \tilde{h}_{Q_2}$
is a quasi-isomorphism in $\Modb \mathcal{P}_{\mathcal{A}}$. The inverse is not true.
\end{prop}

\subsection{Pullback and pushforward}\label{subsection: pullback and pushforward}
Next we consider the pullback and pushforward of (quasi-)cohesive modules. For simplicity we focus on dg-algebras instead of curved dg-algebras. Let $f: \mathcal{A}\to \mathcal{B}$ be a morphism between  dg-algebras. We define a dg functor $f^*: \mathcal{P}_{\mathcal{A}} \to  \mathcal{P}_{\mathcal{B}}$ as follows. Given $(E^{\bullet},\mathbb{E})$ a cohesive module over $\mathcal{A}$, define $f^*(E)$ to be $
(E^{\bullet}\otimes_{\mathcal{A}^0}\mathcal{B}^0,\mathbb{E}_{\mathcal{B}})$
where
$$
\mathbb{E}_{\mathcal{B}}(e\otimes b)=\mathbb{E}(e)b+(-1)^{|e|}e\otimes d_{\mathcal{B}}(b).
$$
We could check that $\mathbb{E}_{\mathcal{B}}$ is still a $\mathbb{Z}$-connection and satisfies $\mathbb{E}_{\mathcal{B}}^2=0$.  The functor $f^*$ on morphisms is defined in the same way. We call $f^*$ the \emph{pullback functor}.

We could define $f^*: \mathcal{C}_{\mathcal{A}} \to  \mathcal{C}_{\mathcal{B}}$ in the same way. Moreover, given composable morphisms of dg-algebras $f$ and $g$, there is a natural equivalence $(f\circ g)^*\Rightarrow g^*f^*$ which satisfies the obvious coherence relation.

For $f: \mathcal{A}\to \mathcal{B}$  a morphism between  dg-algebras, there is also a functor in the opposite direction $f_*:\Modb \mathcal{P}_{\mathcal{B}}\to \Modb\mathcal{P}_{\mathcal{A}}$ and $f_*:\Modb \mathcal{C}_{\mathcal{B}}\to \Modb\mathcal{C}_{\mathcal{A}}$ defined by composing with $f^*$. Now suppose that we are in the special case that the natural map
\begin{equation}\label{equation: tensor of dg-algebras}
\mathcal{B}^0\otimes_{\mathcal{A}^0}\mathcal{A}^{\bullet}\to \mathcal{B}^{\bullet}
\end{equation}
is an isomorphism. Then we will define a  pushforward functor $f_*: \mathcal{C}_{\mathcal{B}}\to  \mathcal{C}_{\mathcal{A}}$ as follows. Let $(\mathcal{Q}^{\bullet},\mathbb{Q})$ be a quasi-cohesive $\mathcal{B}$-module. We consider $\mathcal{Q}^{\bullet}$ as a graded $\mathcal{A}^0$-module via $f$. By the assumption there is an isomorphism
$$
\mathcal{Q}^{\bullet}\otimes_{\mathcal{A}^0}\mathcal{A}^{\bullet}\cong \mathcal{Q}^{\bullet}\otimes_{\mathcal{B}^0}\mathcal{B}^{\bullet}.
$$
Then we define $f_*(\mathcal{Q}^{\bullet},\mathbb{Q})=(\mathcal{Q}^{\bullet},\mathbb{Q})$ where the right hand side are the same as the left hand side but considered as graded $\mathcal{A}^0$-modules and $\mathcal{A}$-module maps. We call $f_*$ the \emph{pushforward functor}.

It is easy to check that $f_*: \mathcal{C}_{\mathcal{B}}\to  \mathcal{C}_{\mathcal{A}}$ and $f_*:\Modb \mathcal{C}_{\mathcal{B}}\to \Modb\mathcal{C}_{\mathcal{A}}$ are compatible via the Yoneda embedding. Moreover, both $f^*$ and $f_*$ are compatible with the forgetful functor $\For: \mathcal{C}_{\mathcal{A}}\to \mathcal{C}_{\mathcal{A}^0}$.

\begin{rmk}\label{rmk: pushforward does not restrict to cohesive modules}
In general, for $(\mathcal{E}^{\bullet},\mathbb{E})$ in $\mathcal{P}_{\mathcal{B}}$, its pushforward $f_*(\mathcal{E}^{\bullet},\mathbb{E})$ is not in  $\mathcal{P}_{\mathcal{A}}$.
\end{rmk}

\subsection{The Serre-Swan theorem}
Let $X$ be a $C^{\infty}$-manifold which is not necessarily compact. Let $C^{\infty}(X)$ be the ring of complex-value $C^{\infty}$-functions on $X$. For a $C^{\infty}$ complex vector bundle $E$ on $X$, let $\Gamma(E)$ be the set of $C^{\infty}$-sections of $E$. It is clear that $\Gamma(E)$  is a $C^{\infty}(X)$-module.

We recall the following result.

\begin{thm}\label{thm: Serre-Swan}[Serre-Swan theorem, \cite{nestruev2020smooth} Theorem 12.29 and 12.32]
Let $X$, $C^{\infty}(X)$, and $\Gamma$ as before. Then $\Gamma$ gives an equivalence of categories from the category of finite dimensional $C^{\infty}$ complex vector bundles on $X$ to the category of 
finitely generated projective $C^{\infty}(X)$-modules.
\end{thm}

We have the following corollary.

\begin{coro}\label{coro: summand of a trivial bundle}
Let $X$ be as before. For any finite dimensional $C^{\infty}$ complex vector bundle $E$ on $X$, there exists a trivial finite dimensional $C^{\infty}$ complex vector bundle $T$ on $X$ such that  $E$ is a $C^{\infty}$-subbundle and hence a direct summand of $T$. Moreover, the rank of $T$ is no greater than $N$, which is an integer depending only on the dimension of $X$ and the rank of $E$.
\end{coro}
\begin{proof}It is obvious by Theorem \ref{thm: Serre-Swan}. The upper bound of the rank of $T$ is given by \cite[Corollary 12.28]{nestruev2020smooth}.
\end{proof}

We will also need the following result.

\begin{prop}\label{prop: kernel of surjective bundle map}
Let $X$ be as before. Let $E$ and $F$ finite dimensional $C^{\infty}$ complex vector bundles on $X$.  For a surjective vector bundle morphism $\phi: E\to F$, its kernel $\ker\phi$ is a $C^{\infty}$-subbundle and hence a direct summand of $E$.
\end{prop}

In this paper we will use the terminologies \emph{finite dimensional $C^{\infty}$ complex vector bundles}  and \emph{finitely generated projective $C^{\infty}(X)$-modules} interchangeably.

\subsection{Cohesive modules on complex manifolds}
Let $X$ be a complex manifold, in this paper we consider the Dolbeault dg-algebra 
$$
\mathcal{A}(X)=(\mathcal{A}^{0,\bullet}(X),\bar{\partial}_X,0)
$$
 where $\mathcal{A}^{0,\bullet}(X)$ is the set of $C^{\infty}$-$(0,\bullet)$-forms on $X$. We have the dg-category of cohesive modules $\mathcal{P}_{\mathcal{A}(X)}$. Let $E=(\mathcal{E}^{\bullet}, \mathbb{E})$ be an object in $\mathcal{P}_{\mathcal{A}(X)}$ where $\mathcal{E}^{\bullet}$ is a bounded graded finitely generated projective $\mathcal{A}^{0,0}(X)=C^{\infty}(X)$-module. By Theorem \ref{thm: Serre-Swan}, $\mathcal{E}^{\bullet}$ corresponds to a bounded graded finite dimensional $C^{\infty}$ vector bundle on $X$. In this viewpoint, $\mathbb{E}$ is a $\bar{\partial}$-$\mathbb{Z}$-connection on the graded vector bundle $\mathcal{E}^{\bullet}$.

In the compact case we have the following theorem.

\begin{thm}\label{thm: cohesive modules give dg-enhancement of coherent sheaves}[\cite{block2010duality} Theorem 4.3]
Let $X$ be a compact complex manifold and  $\mathcal{A}(X)=(\mathcal{A}^{0,\bullet}(X),\bar{\partial}_X,0)$ be the Dolbeault dg-algebra. Then the homotopy category $\text{Ho}(\mathcal{P}_{\mathcal{A}(X)})$ is equivalent to $D^b_{\coh}(X)$, the bounded derived category of complexes of coherent $\mathcal{O}_X$-modules, where $\mathcal{O}_X$ is the sheaf of holomorphic functions on $X$.
\end{thm}

\begin{rmk}
It is well-known that in this case $D^b_{\coh}(X)$ is also equivalent to $D_{\perf}(X)$, the derived category of perfect complexes of  $\mathcal{O}_X$-modules.
\end{rmk}

Recall that a perfect complex of  $\mathcal{O}_X$-modules is a complex $\mathcal{E}^{\bullet}$ of $\mathcal{O}_X$-modules such that there exists an open covering $U_i$ of $X$ such that on each $U_i$ there exists a bounded complex of finitely generated locally free $\mathcal{O}_X$-modules $\mathcal{F}^{\bullet}_i$ together with a quasi-isomorphism $\mathcal{F}^{\bullet}_i\to \mathcal{E}^{\bullet}|_{U_i}$.

When $X$ is non-compact, we need the following concept.

\begin{defi}\label{defi: globally bounded perfect complex}[\cite{chuang2021maurer} Definition 7.2]
Let $X$ be a complex manifold which is not necessarily compact. Let $\mathcal{O}_X$ be the sheaf of holomorphic functions on $X$. A complex  of $\mathcal{O}_X$-modules  $\mathcal{E}^{\bullet}$ is called a \emph{globally bounded perfect complex} if there exists an open covering $U_i$ of $X$ and integers $a<b$ and $N>0$ such that on each $U_i$ there exists a bounded complex of finitely generated locally free $\mathcal{O}_X$-modules $\mathcal{F}^{\bullet}_i$ which is concentrated in degrees $[a, b]$ and each $\mathcal{F}^j_i$ has rank $\leq N$, together with  a quasi-isomorphism $\mathcal{F}^{\bullet}_i\to \mathcal{E}^{\bullet}|_{U_i}$. We denote the derived category of globally bounded perfect complexes on $X$ by $D^{\text{B}}_{\perf}(X)$.
\end{defi}

It is clear that when $X$ is compact, any perfect complex on $X$ is globally bounded. However, this is no longer true for non-compact $X$. See \cite[Remark 7.2]{chuang2021maurer}.

\begin{thm}\label{thm: cohesive modules give dg-enhancement of coherent sheaves non-compact}[\cite{chuang2021maurer} Theorem 8.3]
Let $X$ be a  complex manifold and  $\mathcal{A}(X)=(\mathcal{A}^{0,\bullet}(X),\bar{\partial}_X,0)$ be the Dolbeault dg-algebra. Then the homotopy category $\text{Ho}(\mathcal{P}_{\mathcal{A}(X)})$ is equivalent to $D^{\text{B}}_{\perf}(X)$, the derived category of globally bounded perfect complexes on $X$.
\end{thm}

\subsection{The dg-presheaf $\mathcal{P}_{\mathcal{A}}$ and the descent problem}
Let $U\subset X$ be an open subset of $X$ and we define the dg-category of (quasi-)cohesive modules on $U$ as follows.
\begin{defi}\label{defi: cohesive modules on open subset}
Let $U\subset X$ be an open subset of $X$. We define the dg-algebra $\mathcal{A}(U)$ to be
$$
\mathcal{A}(U)=(\mathcal{A}^{0,\bullet}(U),\bar{\partial},0)
$$
Then we could define the dg-categories $\mathcal{P}_{\mathcal{A}(U)}$, $\mathcal{P}_{\mathcal{A}^0(U)}$, $\mathcal{C}_{\mathcal{A}(U)}$, and $\mathcal{C}_{\mathcal{A}^0(U)}$.
\end{defi}

For an inclusion $U\subset V$ of open subsets we have the restriction map $r: \mathcal{A}(V)\to\mathcal{A}(U)$. Hence we get the pullback functor $r^*: \mathcal{P}_{\mathcal{A}(V)}\to \mathcal{P}_{\mathcal{A}(U)}$ as in Section \ref{subsection: pullback and pushforward}. Therefore the assignment
$$
U\mapsto \mathcal{P}_{\mathcal{A}(U)}
$$
gives a dg-presheaf on $X$ and we denote it by $\mathcal{P}_{\mathcal{A}}$.

For an open cover $\mathcal{U}=\{U_i\}$ of $X$, its \emph{\v{C}ech nerve} is a simplicial space
$$
\begin{tikzcd}
\cdots  \arrow[yshift=1.2ex]{r}\arrow[yshift=0.4ex]{r}\arrow[yshift=-0.4ex]{r}\arrow[yshift=-1.2ex]{r}& \coprod U_i\cap U_j\cap U_k \arrow[yshift=1ex]{r}\arrow{r}\arrow[yshift=-1ex]{r}& \coprod U_i\cap U_j  \arrow[yshift=0.7ex]{r}\arrow[yshift=-0.7ex]{r}  &  \coprod U_i.
\end{tikzcd}
$$
and we consider the resulting cosimplicial diagram of dg-categories
\begin{equation}\label{equation: cosimplicial diagram of cohesive modules of an open cover}
\begin{tikzcd}
\prod \mathcal{P}_{\mathcal{A}}(U_i) \arrow[yshift=0.7ex]{r}\arrow[yshift=-0.7ex]{r}& \prod \mathcal{P}_{\mathcal{A}}(U_i\cap U_j) \arrow[yshift=1ex]{r}\arrow{r}\arrow[yshift=-1ex]{r}  &  \prod \mathcal{P}_{\mathcal{A}}(U_i\cap U_j\cap U_k)  \arrow[yshift=1.2ex]{r}\arrow[yshift=0.4ex]{r}\arrow[yshift=-0.4ex]{r}\arrow[yshift=-1.2ex]{r}&\cdots
\end{tikzcd}
\end{equation}

It is clear that the descent data of $\mathcal{P}_{\mathcal{A}}$ with respect to the open cover $\{U_i\}$ is given by the homotopy limit of Diagram (\ref{equation: cosimplicial diagram of cohesive modules of an open cover}) in $\text{DgCat}_{\text{DK}}$, the category of all dg-categories with the Dwyer-Kan model structure. In Section \ref{section: review of twisted complex} we will present this homotopy limit as the dg-category of twisted complexes. The main topic of this paper is to prove that $\mathcal{P}_{\mathcal{A}}(X)$ is quasi-equivalent to the homotopy limit of Diagram (\ref{equation: cosimplicial diagram of cohesive modules of an open cover}).

We will also need the following result on pushforward.

\begin{prop}\label{prop: pushforward for open subset}
For an inclusion $U\subset V$ of open subsets let $r: \mathcal{A}(V)\to\mathcal{A}(U)$ be the restriction map. Then we can define the pushforward functor 
$$
r_*: \mathcal{C}_{\mathcal{A}(U)}\to \mathcal{C}_{\mathcal{A}(V)}.
$$
\end{prop}
\begin{proof}
Recall that $\mathcal{A}^\bullet(U)$ is the set of $C^{\infty}$-$(0,\bullet)$-forms on $U$. From the definition of differential forms it is clear that the natural homomorphism
$$
\mathcal{A}^0(U)\otimes_{\mathcal{A}^0(V)}\mathcal{A}^\bullet(V)\to \mathcal{A}^\bullet(U)
$$
is an isomorphism. The claim of the proposition follows from the construction in Section \ref{subsection: pullback and pushforward}.
\end{proof}

\section{Twisted complexes}\label{section: review of twisted complex}
\subsection{Definition and basic facts}
Toledo and Tong \cite{toledo1978duality} introduced  \emph{twisted complexes} in the 1970's as a way to obtain global resolutions of perfect complexes of sheaves on a complex manifold. In 2015 Wei proved in \cite{wei2016twisted} that the dg-category of twisted perfect complexes give a dg-enhancement of the derived category of perfect complexes.

In this paper we give a slightly generalized definition of twisted complexes so that we could apply it in the descent problem of $\mathcal{P}_{\mathcal{A}}$. For reference of twisted complexes see \cite{o1981trace} Section 1 or \cite{wei2016twisted} Section 2.

Let $X$ be a  paracompact topological space and $\mathfrak{F}$ be a dg-presheaf on $X$. Let $\mathcal{U}=\{U_i\}$ be an locally finite open cover of $X$. Let $U_{i_0\ldots i_n}$ denote the intersection $U_{i_0}\cap\ldots \cap U_{i_n}$.

Let $\{E_i\}$ and $\{F_i\}$ be two collections of objects in $\mathfrak{F}(U_i)$ for each $U_i$.  We can consider the map
\begin{equation}\label{equation: map with bigrade between graded sheaves}
C^{\bullet}(\mathcal{U},\text{Hom}^{\bullet}(E,F))=\bigoplus_{p,q}C^p(\mathcal{U},\text{Mor}^q_{\mathfrak{F}}(E,F)).
\end{equation}
An element $u^{p,q}$ of $C^p(\mathcal{U},\text{Mor}^q_{\mathfrak{F}}(E,F))$ gives an element $u^{p,q}_{i_0\ldots i_p}$ of $\text{Mor}^q_{\mathfrak{F}}(E_{i_p},F_{i_0})$ over each non-empty intersection $U_{i_0\ldots i_n}$. Notice that we require $u^{p,q}$ to be a morphism from the $E$ on the last subscript of $U_{i_0\ldots i_n}$ to the $F$ on the first subscript of $U_{i_0\ldots i_n}$.

We need to define the compositions of $C^{\bullet}(\mathcal{U},\text{Mor}^{\bullet}_{\mathfrak{F}}(E,F))$. Let $\{G_i\}$ be a third collection of objects. There is a composition map
$$
C^{\bullet}(\mathcal{U},\text{Mor}^{\bullet}(F,G)) \times C^{\bullet}(\mathcal{U},\text{Mor}^{\bullet}(E,F))\rightarrow C^{\bullet}(\mathcal{U},\text{Mor}^{\bullet}(E,G)).
$$
In fact, for $u^{p,q}\in C^p(\mathcal{U},\text{Mor}^q(F,G))$ and $v^{r,s} \in C^{r}(\mathcal{U},\text{Mor}^{s}(E,F))$, their composition $(u\cdot v)^{p+r,q+s}$ is given by (see \cite{o1981trace} Equation (1.1))
\begin{equation}\label{equation: composition of maps between graded sheaves}
(u\cdot v)^{p+r,q+s}_{i_0\ldots i_{p+r}}=(-1)^{qr}u^{p,q}_{i_0\ldots i_p}v^{r,s}_{i_p\ldots i_{p+r}}
\end{equation}
where the right hand side is the composition of sheaf maps.

In particular $C^{\bullet}(\mathcal{U},\text{Mor}^{\bullet}(E,E))$ becomes an associative algebra under this composition (It is easy but tedious to check the associativity).

There is also a \v{C}ech-style differential operator $\delta$ on $C^{\bullet}(\mathcal{U},\text{Mor}^{\bullet}(E,F))$ and of bidegree $(1,0)$ given by the formula
\begin{equation}\label{equation: delta on maps}
(\delta u)^{p+1,q}_{i_0\ldots i_{p+1}}=\sum_{k=1}^p(-1)^k u^{p,q}_{i_0\ldots \widehat{i_k} \ldots i_{p+1}}|_{U_{i_0\ldots i_{p+1}}} \,\text{ for } u^{p,q}\in C^p(\mathcal{U},\text{Mor}^q_{\mathfrak{F}}(E,F))
\end{equation}
It is not difficult to check that the \v{C}ech differential satisfies the Leibniz rule.

Now we can introduce the definition of twisted complexes.

\begin{defi}\label{defi: twisted complex}
Let $X$ be a paracompact topological space and $\mathfrak{F}$ be a dg-presheaf on $X$. Let $\mathcal{U}=\{U_i\}$ be an locally finite open cover of $X$. A \emph{twisted  complex} consists of a collection objects $E_i$ of $\mathfrak{F}(U_i)$ together with a collection of morphisms
$$
a=\sum_{k \geq 0} a^{k,1-k}
$$
where $a^{k,1-k}\in C^k(\mathcal{U},\text{Mor}^{1-k}(E,E)) $  which satisfies the equation
\begin{equation}\label{equation: MC for twisted complex}
\delta a+ a\cdot a=0.
\end{equation}
More explicitly, for $k\geq 0$
\begin{equation}\label{equation: MC for twisted complex explicit}
\delta a^{k-1,2-k}+ \sum_{i=0}^k a^{i,1-i}\cdot a^{k-i,1-k+i}=0.
\end{equation}

We impose two additional requirements on $a$:
\begin{enumerate}
\item For any $u^{p,q}\in \text{Mor}^q_{\mathfrak{F}}(E_{i_p},E_{i_0})$, the assignment
$$
u^{p,q}\mapsto (-1)^p[a^{0,1}_{i_0}\cdot u^{p,q}-(-1)^{p+q}u^{p,q}\cdot a^{0,1}_{i_p}]\in \text{Mor}^{q+1}_{\mathfrak{F}}(E_{i_p},E_{i_0})
$$
coincides with the differential in the dg-category $\mathfrak{F}(U_{i_0\ldots i_p})$;

\item $a^{1,0}_{ii}\in \text{Mor}^0_{\mathfrak{F}}(E_i,E_i)$ is invertible up to homotopy.
\end{enumerate}

Twisted  complexes on $(X,\mathfrak{F}, \{U_i\})$ form a dg-category: the objects are   twisted  complexes $(E_i,a)$ and the morphism from $\mathcal{E}=(E_i,a)$ to $\mathcal{F}=(F_i,b)$ are $C^{\bullet}(\mathcal{U},\text{Mor}^{\bullet}(E,F))$ with the total degree. Moreover, the differential on a morphism $\phi$ is given by
\begin{equation}\label{equation: differential on morphisms of twisted complexes}
d \phi=\delta \phi+b\cdot \phi-(-1)^{|\phi|}\phi\cdot a.
\end{equation}

We denote the dg-category of twisted complexes on $(X,\mathfrak{F}, \{U_i\})$ by $\text{Tw}(X, \mathfrak{F}, U_i)$. If there is no danger of confusion we can simply denote it by $\text{Tw}(X)$.
\end{defi}

Here we list some special cases of twisted complexes for various dg-presheaves $\mathfrak{F}$:
\begin{enumerate}
\item Let $(X,\mathcal{R})$ be a ringed space and $\mathfrak{F}=\text{Cpx}$ be the dg-presheaf which assigns to each open subspace $U$ the dg-category of
complexes of left $\mathcal{R}$-modules on $U$, then the dg-category of twisted complexes $\text{Tw}(X, \mathfrak{F}, U_i)$ as in Definition \ref{defi: twisted complex} is exactly the dg-category of twisted complexes $\text{Tw}(X, \mathcal{R}, U_i)$  as in \cite{wei2016twisted} Definition 2.12.

\item Again let $(X,\mathcal{R})$ be a ringed space and $\mathfrak{F}=\Perf$ be the dg-presheaf which assigns to each open subspace $U$ the dg-category of
bounded complexes of finitely generated locally free left $\mathcal{R}$-modules on $U$, then the dg-category of twisted complexes $\text{Tw}(X, \mathfrak{F}, U_i)$ as in Definition \ref{defi: twisted complex} is exactly the dg-category of twisted perfect complexes $\text{Tw}_{\perf}(X, \mathcal{R}, U_i)$  as in \cite{wei2016twisted} Definition 2.14, which is also called \emph{twisted cochain} in \cite{o1981trace}.

\item Let $X$ be a complex manifold and $\mathfrak{F}=\mathcal{P}_{\mathcal{A}}$. The  dg-category of twisted complexes $\text{Tw}(X, \mathcal{P}_{\mathcal{A}}, U_i)$ is the main subject of this paper.

\item Let $X$ be a complex manifold and $\mathfrak{F}=\mathcal{P}_{\mathcal{A}^0}$. Then the dg-category of twisted complexes $\text{Tw}(X, \mathcal{P}_{\mathcal{A}^0}, U_i)$ is the same as $\text{Tw}_{\text{perf}}(X, \mathcal{A}^0, U_i)$ and we will further study it in Section \ref{section: the gluing of underlying complexes}.

\item Let $X$ be a complex manifold and $\mathfrak{F}=\mathcal{C}_{\mathcal{A}}$ or $\mathcal{C}_{\mathcal{A}^0}$, the dg-presheaf of quasi-cohesive modules.  The resulting dg-categories $\text{Tw}(X, \mathcal{C}_{\mathcal{A}}, U_i)$ and $\text{Tw}(X, \mathcal{C}_{\mathcal{A}^0}, U_i)$ play auxiliary roles in this paper.
\end{enumerate}

The importance of twisted complexes in descent theory is illustrated by the following theorem.

\begin{thm}\label{thm: twisted complexes give homotopy limit}[\cite{block2017explicit}, \cite{arkhipov2021homotopy}]
Let $\mathfrak{F}$ be a dg-presheaf. The dg-category $\text{Tw}(X, \mathfrak{F}, U_i)$ is quasi-equivalent to the homotopy limit of the cosimplicial diagram
$$
\begin{tikzcd}
\prod \mathfrak{F}(U_i) \arrow[yshift=0.7ex]{r}\arrow[yshift=-0.7ex]{r}& \prod \mathfrak{F}(U_i\cap U_j) \arrow[yshift=1ex]{r}\arrow{r}\arrow[yshift=-1ex]{r}  &  \prod \mathfrak{F}(U_i\cap U_j\cap U_k)  \arrow[yshift=1.2ex]{r}\arrow[yshift=0.4ex]{r}\arrow[yshift=-0.4ex]{r}\arrow[yshift=-1.2ex]{r}&\cdots
\end{tikzcd}
$$
$\square$
\end{thm}

It is clear that Theorem \ref{thm: twisted complexes give homotopy limit} applies to all above cases.

\begin{rmk}
Theorem \ref{thm: twisted complexes give homotopy limit} is proved in \cite{block2017explicit} under the condition that $\mathfrak{F}$  sends finite coproducts to
products. In \cite{arkhipov2021homotopy} this condition is removed and the authors prove the result for arbitrary cosimplicial diagram of dg-categories.
\end{rmk}

In \cite[Section 2.5]{wei2016twisted} it has been shown that $\text{Tw}(X, \mathfrak{F}, U_i)$ has a pre-triangulated structure for all above cases, hence Ho$\text{Tw}(X, \mathfrak{F}, U_i)$ is a triangulated category.  In more details we have the following definitions.

\begin{defi}\label{defi: shift of twisted complex}[Shift]
Let $\mathcal{E}=(E^{\bullet}_i,a)$ be a twisted  complex. We define its shift   $\mathcal{E}[1]$ to be $\mathcal{E}[1]=(E[1]^{\bullet}_i,a[1])$ where
$$
E[1]^{\bullet}_i=E^{\bullet+1}_i \text{ and } a[1]^{k,1-k}=(-1)^{k-1}a^{k,1-k}.
$$

Moreover, let $\phi: \mathcal{E}\to \mathcal{F}$ be a morphism. We define its shift $\phi[1]$ as
$$
\phi[1]^{p,q}=(-1)^q\phi^{p,q}.
$$
\end{defi}

\begin{defi}\label{defi: mapping cone}[Mapping cone]
Let $\phi^{\bullet,-\bullet}$ be a closed degree zero map between twisted perfect complexes $\mathcal{E}=(E^{\bullet},a^{\bullet,1-\bullet})$ and $\mathcal{F}=(F^{\bullet},b^{\bullet,1-\bullet})$ , we can define the \emph{mapping cone} $\mathcal{G}=(G,c)$ of $\phi$ as follows (see \cite[Section 1.1]{o1985grothendieck}):
$$
G^n_i:=E^{n+1}_i\oplus F^n_i
$$
and
\begin{equation}\label{equation: diff in mapping cone}
c^{k,1-k}_{i_0\ldots i_k}=\begin{pmatrix}(-1)^{k-1} a^{k,1-k}_{i_0\ldots i_k}&0\\ (-1)^k\phi^{k,-k}_{i_0\ldots i_k}&b^{k,1-k}_{i_0\ldots i_k}\end{pmatrix}.
\end{equation}
\end{defi}

\subsection{Quasi-isomorphisms between twisted complexes}
For $\mathfrak{F}=\mathcal{C}_{\mathcal{A}}$ or $\mathcal{C}_{\mathcal{A}^0}$, we have the following definition of quasi-isomorphism between twisted complexes.

\begin{defi}\label{defi: quasi-isomorphism between twisted complexes}
Let $\phi:\mathcal{E}\to\mathcal{F}$ be a degree zero closed morphism in $\text{Tw}(X, \mathcal{C}_{\mathcal{A}^0}, U_i)$. Then we call $\phi$ a \emph{quasi-isomorphism} if and only if its $(0,0)$-component
$$
\phi^{0,0}:(E^{\bullet}_i,a^{0,1}_i)\to (F^{\bullet}_i,b^{0,1}_i)
$$
is a quasi-isomorphism of complexes of $\mathcal{A}^0(U_i)$-modules for each $i$.

Moreover, let $\phi:\mathcal{E}\to\mathcal{F}$ be a degree zero closed morphism in $\text{Tw}(X, \mathcal{C}_{\mathcal{A}}, U_i)$. Then we call $\phi$ a quasi-isomorphism if and only if $\For(\phi): \For(\mathcal{E})\to\For(\mathcal{F})$ is a quasi-equivalence in $\text{Tw}(X, \mathcal{C}_{\mathcal{A}^0}, U_i)$, where $\For$ is the forgetful functor.
\end{defi}

\begin{rmk}
A quasi-isomorphism is called a weak equivalence in \cite[Definition 2.27]{wei2016twisted}.
\end{rmk}

\begin{lemma}\label{lemma: lift under quasi-isomorphism in twisted complex}
Let $\mathcal{E}$ be an object in  $\text{Tw}(X, \mathcal{P}_{\mathcal{A}}, U_i)$ and $\mathcal{F}$ and $\mathcal{G}$ be two objects in $\text{Tw}(X, \mathcal{C}_{\mathcal{A}}, U_i)$. Let
$\phi:\mathcal{E}\to \mathcal{G}$ be a degree zero closed morphism in $\text{Tw}(X, \mathcal{C}_{\mathcal{A}}, U_i)$ and $\psi: \mathcal{F}\to \mathcal{G}$ be a quasi-isomorphism in $\text{Tw}(X, \mathcal{C}_{\mathcal{A}}, U_i)$. Then $\phi$ could be lifted to a closed degree zero morphism $\eta: \mathcal{E}\to \mathcal{F}$ up to homotopy, i.e. there exists an $\eta: \mathcal{E}\to \mathcal{F}$ such that $\psi\circ \eta=\phi$ up to homotopy. The same result holds for $\text{Tw}(X, \mathcal{C}_{\mathcal{A}^0}, U_i)$.
\end{lemma}
\begin{proof}
It is a standard spectral sequence argument. See \cite[Lemma 2.30]{wei2016twisted}.
\end{proof}

We have some further results on quasi-isomorphisms if both objects are in $\text{Tw}(X, \mathcal{P}_{\mathcal{A}}, U_i)$.

\begin{prop}\label{prop: quasi-isomorphism of twisted complexes  in PA is homotopy equivalence}
Let $\mathcal{E}$ and $\mathcal{F}$ be objects in $\text{Tw}(X, \mathcal{P}_{\mathcal{A}}, U_i)$. Then a degree zero closed morphism $\phi: \mathcal{E}\to\mathcal{F}$ is a quasi-isomorphism if and only if $\phi$ is a homotopy equivalence. The same result holds for $\text{Tw}(X, \mathcal{P}_{\mathcal{A}^0}, U_i)$.
\end{prop}
\begin{proof}
By Definition \ref{defi: quasi-isomorphism between twisted complexes}, $\phi^{0,0}_i:\mathcal{E}_i\to \mathcal{F}_i$ is a quasi-isomorphism in $\mathcal{P}_{\mathcal{A}(U_i)}$. Then by Proposition \ref{prop: homotopy equivalences in PA}, we have its homotopy inverse $\psi_i:\mathcal{F}_i\to \mathcal{E}_i$ in $\mathcal{P}_{\mathcal{A}(U_i)}$. By a simple spectral sequence
argument which is the same as the proof of \cite{block2010duality} Proposition 2.9, we could extend $\psi_i$ to a degree zero closed morphism in $\text{Tw}(X, \mathcal{P}_{\mathcal{A}^0}, U_i)$. A similar argument works for $\text{Tw}(X, \mathcal{P}_{\mathcal{A}^0}, U_i)$. See also \cite[Proposition 2.31]{wei2016twisted}.
\end{proof}

\subsection{Globally bounded twisted complexes}\label{subsec: globally bounded twisted}
When $\mathfrak{F}=\Perf$, $\mathcal{P}_{\mathcal{A}}$, or $\mathcal{P}_{\mathcal{A}^0}$, we have the following dg-subcategory of $\text{Tw}(X, \mathfrak{F}, U_i)$

\begin{defi}\label{defi: globally bounded twisted complexes}
Let $\mathfrak{F}$ be $\Perf$, $\mathcal{P}_{\mathcal{A}}$, or $\mathcal{P}_{\mathcal{A}^0}$. A twisted complex $(E_i, a)$ in $\text{Tw}(X, \mathfrak{F}, U_i)$ is called \emph{globally bounded} if there exist integers $a<b$ and $N>0$ such that on each $U_i$ the underline complex $E^{\bullet}_i$ of $\mathfrak{F}(U_i)$ is concentrated in degrees $[a,b]$ and each of the $E^k_i$'s has rank $\leq N$.

Globally bounded twisted complexes form a full dg-subcategory of $\text{Tw}(X, \mathfrak{F}, U_i)$ and we denote it by $\text{Tw}^{\text{B}}(X, \mathfrak{F}, U_i)$.
\end{defi}

It is clear that $\text{Tw}^{\text{B}}(X, \mathfrak{F}, U_i)$ inherits the pre-triangulated structure from $\text{Tw}(X, \mathfrak{F}, U_i)$. Moreover Lemma \ref{lemma: lift under quasi-isomorphism in twisted complex} and Proposition \ref{prop: quasi-isomorphism of twisted complexes  in PA is homotopy equivalence} applies to $\text{Tw}^{\text{B}}(X, \mathcal{P}_{\mathcal{A}}, U_i)$ as well.

\begin{rmk}
If the open cover $\{U_i\}$ is finite, then $\text{Tw}^{\text{B}}(X, \mathfrak{F}, U_i)$ coincides with $\text{Tw}(X, \mathfrak{F}, U_i)$.
\end{rmk}

\begin{rmk}
In general, the analogue of the claim in  
Theorem \ref{thm: twisted complexes give homotopy limit} does not hold for $\text{Tw}^{\text{B}}$ unless the cover  $\{U_i\}$ is finite.
\end{rmk}

\begin{rmk}
Using the same method as in \cite{wei2016twisted}, we can prove that the dg-category $\text{Tw}^{\text{B}}(X, \Perf, U_i)$ gives a dg-enhancement of the derived category of globally bounded perfect complexes   $D^{\text{B}}_{\perf}(X)$. Nevertheless we do not need this result in this paper.
\end{rmk}



\subsection{The twisting functor and the sheafification functor}
For $\mathfrak{F}=\mathcal{C}_{\mathcal{A}}$ or $\mathcal{C}_{\mathcal{A}^0}$, we  define a pair of adjoint dg-functors
\begin{equation}
\mathcal{T}: \mathcal{C}_{\mathcal{A}}(X)\rightleftarrows \text{Tw}(X, \mathcal{C}_{\mathcal{A}}, U_i): \mathcal{S}
\end{equation}
and study their properties in this subsection.

First we define the natural dg-functor  $\mathcal{T}: \mathcal{C}_{\mathcal{A}}(X)\to \text{Tw}(X, \mathcal{C}_{\mathcal{A}}, U_i)$
\begin{defi}[\cite{wei2016twisted} Definition 3.11]\label{defi: twisted functor}
Let $(\mathcal{Q},\mathbb{Q})$ be an object in $\mathcal{C}_{\mathcal{A}}(X)$. We define its associated twisted  complex $\mathcal{T}(Q)\in \text{Tw}(X, \mathcal{C}_{\mathcal{A}}, U_i)$  by restricting to the $U_i$'s. In more details, we define $(E^{\bullet},a)=\mathcal{T}(Q)$ as
$$
E^{n}_i=\mathcal{Q}^n|_{U_i}
$$
and
$$
a^{0,1}_i=\mathbb{Q}|_{U_i},~a^{1,0}_{ij}=\id_{\mathcal{Q}^{\bullet}|_{U_{ij}}} \text{ and } a^{k,1-k}=0 \text{ for }k\geq 2.
$$
The $\mathcal{T}$ of morphisms is defined in a similar way.
We call the dg-functor $\mathcal{T}: \mathcal{C}_{\mathcal{A}}(X)\to \text{Tw}(X, \mathcal{C}_{\mathcal{A}}, U_i)$ the \emph{twisting functor}. We can define $\mathcal{T}: \mathcal{C}_{\mathcal{A}^0}(X)\to \text{Tw}(X, \mathcal{C}_{\mathcal{A}^0}, U_i)$ in the same way.
\end{defi}

The definition of $\mathcal{S}: \text{Tw}(X, \mathcal{C}_{\mathcal{A}}, U_i)\to \mathcal{C}_{\mathcal{A}}(X)$ is more complicated.
First we noticed that a  twisted complex $\mathcal{E}=(E^{\bullet}_i,a)$ is not a globally defined quasi-cohesive complex  on $X$. Nevertheless in this subsection we associate a global complex  to each twisted complex.

Let $E_i$ be an object in  $\mathcal{C}_{\mathcal{A}(U_i)}$. By Proposition \ref{prop: pushforward for open subset}, we could use the pushforward $r_*$ to treat $E_i$ as an object in $\mathcal{C}_{\mathcal{A}(X)}$. Moreover, for $E_{i_0}$  in  $\mathcal{C}_{\mathcal{A}(U_{i_0})}$, we could first restrict $E_{i_0}$ to $\mathcal{C}_{\mathcal{A}(U_{i_0\ldots i_k})}$ and then pushforward to $\mathcal{C}_{\mathcal{A}(X)}$.

\begin{defi}\label{defi: sheaf associated to twisted complex}[\cite[Definition 3.1]{wei2016twisted}]
For a twisted complex of quasi-cohesive modules $\mathcal{E}=(E_i,a)$, we define the associated quasi-cohesive module $\mathcal{S}(\mathcal{E})$ on $X$ as follows: for each $n$, the degree $n$ component $\mathcal{S}^n(\mathcal{E})$ is an $\mathcal{A}^0(X)$-module
$$
\mathcal{S}^n(\mathcal{E}):=\prod_{p+q=n}\prod_{i_0\ldots i_p}E^q_{i_0}|_{U_{i_0\ldots i_p}}
$$
where the right hand side is considered as an $\mathcal{A}^0(X)$-module by pushforward.

The connection on $\mathcal{S}^{\bullet}(\mathcal{E})$ is defined to be  of $\mathbb{S}(\mathcal{E})=\delta+a$ considered as morphisms on $X$.

It is obvious that $(\mathcal{S}^{\bullet}(\mathcal{E}),\mathbb{S}(\mathcal{E}))$ is a quasi-cohesive module in  $\mathcal{C}_{\mathcal{A}}(X)$. The functor $\mathcal{S}$ on morphisms is defined in the same way. We call $\mathcal{S}$ the \emph{sheafification functor}.
\end{defi}

\begin{rmk}
The  functor $\mathcal{S}$ for $\mathcal{P}_{\mathcal{A}}$ is a generalization of the functor $\tilde{A}$ in \cite[Definition 6.2]{ben2013milnor}, and $\mathcal{S}$ for $\mathcal{P}_{\mathcal{A}^0}$ is a generalization of the functor $\tilde{\psi}$ in \cite[Equation (5.6)]{ben2013milnor}.
\end{rmk}

\begin{rmk}
It is clear that $\mathcal{T}$ restricts to a functor $\mathcal{T}: \mathcal{P}_{\mathcal{A}}(X)\to \text{Tw}(X, \mathcal{P}_{\mathcal{A}}, U_i)$ as well as $\mathcal{T}: \mathcal{P}_{\mathcal{A}^0}(X)\to \text{Tw}(X, \mathcal{P}_{\mathcal{A}^0}, U_i)$. On the other hand, the image of the functor $\mathcal{S}: \text{Tw}(X, \mathcal{P}_{\mathcal{A}}, U_i)\to \mathcal{C}_{\mathcal{A}}(X)$ is not contained in $\mathcal{P}_{\mathcal{A}}(X)$.
\end{rmk}

\begin{prop}\label{prop: adjunction of twisted functor and sheafification functor}
$$
\mathcal{T}: \mathcal{C}_{\mathcal{A}}(X)\rightleftarrows \text{Tw}(X, \mathcal{C}_{\mathcal{A}}, U_i): \mathcal{S}
$$
is a pair of adjoint functors. Moreover, the unit morphism of the adjunction $\epsilon(E): E\to \mathcal{S}\circ\mathcal{T}(E)$ is a quasi-isomorphism (in the sense of Definition \ref{defi: quasi-isomorphism between quasi-cohesive modules}) for any object $E\in \mathcal{C}_{\mathcal{A}}(X)$.  The same results applies to
$\mathcal{T}: \mathcal{C}_{\mathcal{A}^0}(X)\rightleftarrows \text{Tw}(X, \mathcal{C}_{\mathcal{A}^0}, U_i): \mathcal{S}$.
\end{prop}
\begin{proof}
It is a routine check.
\end{proof}

Moreover, for a refinement $\{V_j\}$ of the open cover $\{U_i\}$, we could define the twisted functor and the sheafification functor in the same way.

\begin{prop}\label{prop: adjunction of twisted functor and sheafification functor refinement}
$$
\mathcal{T}: \text{Tw}(X, \mathcal{C}_{\mathcal{A}}, U_i)\rightleftarrows \text{Tw}(X, \mathcal{C}_{\mathcal{A}}, V_j): \mathcal{S}
$$
is a pair of adjoint functors. Moreover, the unit morphism of the adjunction $\epsilon(E): E\to \mathcal{S}\circ\mathcal{T}(E)$ is a quasi-isomorphism (in the sense of Definition \ref{defi: quasi-isomorphism between twisted complexes}) for any object $E\in \text{Tw}(X, \mathcal{C}_{\mathcal{A}}, U_i)$.  The same results applies to
$\mathcal{T}: \text{Tw}(X, \mathcal{C}_{\mathcal{A}^0}, U_i)\rightleftarrows \text{Tw}(X, \mathcal{C}_{\mathcal{A}^0}, V_j): \mathcal{S}$.
\end{prop}
\begin{proof}
It is a routine check.
\end{proof}

The goal of this paper is to prove that $\mathcal{T}: \mathcal{P}_{\mathcal{A}}(X)\to \text{Tw}(X, \mathcal{P}_{\mathcal{A}}, U_i)$ is a quasi-equivalence of dg-categories.

\section{The gluing of underlying complexes}\label{section: the gluing of underlying complexes}

In this section we study in more details of the adjunction
 $$
\mathcal{T}: \mathcal{C}_{\mathcal{A}^0}(X)\rightleftarrows \text{Tw}(X, \mathcal{C}_{\mathcal{A}^0}, U_i): \mathcal{S}
$$

\begin{lemma}\label{lemma: sheafification is perfect}
For a twisted perfect complex $\mathcal{E}=(E_i,a)$, its sheafification $\mathcal{S}(\mathcal{E})$ is a perfect complex of $\mathcal{A}^0$-modules on $X$. If in addition $\mathcal{E}=(E_i,a)$ is globally bounded, then $\mathcal{S}(\mathcal{E})$ is also globally bounded.
\end{lemma}
\begin{proof}
See \cite[Proposition 3.5 and Corollary 3.8]{wei2016twisted}
\end{proof}

We also need the following result.

\begin{prop}\label{prop: glue of smooth perf complex}
 Let $X$ be a $C^{\infty}$-manifold with $\mathcal{A}^0$ the sheaf of $C^{\infty}$-functions. For any globally bounded perfect complex of $\mathcal{A}^0$-modules $\mathcal{P}$, there exists a bounded complex of finitely generated locally free $\mathcal{A}^0$-modules $\mathcal{R}$ together with a quasi-isomorphism of  complexes of $\mathcal{A}^0$-modules $\phi: \mathcal{R}\to \mathcal{P}$.
\end{prop}
\begin{proof}
The proof is essentially the same of the proof of \cite[Lemma 7.5]{chuang2021maurer}.
\end{proof}

\begin{prop}\label{prop: there is a complex of vector bundles for the sheafification}
 Let $X$ be a $C^{\infty}$-manifold with $\mathcal{A}^0$ the sheaf of $C^{\infty}$-functions. Let  $\{U_i\}$ be a locally finite open cover of $X$. Then for every globally bounded twisted  complex $\mathcal{F}=(F^{\bullet}_i, b)\in \text{Tw}^{\text{B}}(X, \mathcal{P}_{\mathcal{A}^0}, U_i)$, there is an object $\mathcal{E}=(E^{\bullet},d)\in \mathcal{P}_{\mathcal{A}^0(X)}$  together with a quasi-isomorphism  $\phi: \mathcal{E}\to \mathcal{S}(\mathcal{F})$ in $\mathcal{P}_{\mathcal{A}^0(X)}$.

 Moreover, the corresponding morphism $\eta_{\mathcal{F}}\circ \mathcal{T}(\phi):\mathcal{T}(\mathcal{E})\to \mathcal{F}$ is a homotopy equivalence in $\text{Tw}(X, \mathcal{P}_{\mathcal{A}^0}, U_i)$.
 \end{prop}

\begin{proof}
 By Lemma \ref{lemma: sheafification is perfect},  $\mathcal{S}(\mathcal{F})$ is a globally bounded  perfect complex of $\mathcal{A}^0$-modules on $X$. Then by Proposition \ref{prop: glue of smooth perf complex}, there exists a bounded complex of finitely generated locally free $\mathcal{A}^0$-modules $\mathcal{E}$ together with a quasi-isomorphism of  complexes of $\mathcal{A}^0$-modules $\phi: \mathcal{E}\to \mathcal{S}(\mathcal{F})$. This proves the first part of the proposition.

We can then consider $\mathcal{T}(\phi): \mathcal{T}(\mathcal{E})\to \mathcal{T}\mathcal{S}(\mathcal{F})$. By definition $\mathcal{T}(\phi)$ is a quasi-isomorphism as in Definition \ref{defi: quasi-isomorphism between twisted complexes}. By \cite[Proposition 3.13]{wei2016twisted} , $\eta_{\mathcal{F}}:  \mathcal{T}\mathcal{S}(\mathcal{F})\to \mathcal{F}$ is also a quasi-isomorphism, so is the composition $\eta_{\mathcal{F}}\circ \mathcal{T}(\phi):\mathcal{T}(\mathcal{E})\to \mathcal{F}$. Since $\mathcal{T}(\mathcal{E})$ and $\mathcal{F}$ are both in  $\mathcal{P}_{\mathcal{A}^0(X)}$, by Proposition \ref{prop: quasi-isomorphism of twisted complexes  in PA is homotopy equivalence} $\eta_{\mathcal{F}}\circ \mathcal{T}(\phi)$ is a homotopy equivalence.
\end{proof}

\section{The descent of cohesive modules}\label{section: the descent of cohesive modules}

We first recall the following theorem on objects in $\mathcal{P}_{\mathcal{A}}$.

\begin{thm}\label{thm: quasi-representability of h tilde X}[\cite{block2010duality} Theorem 3.13]
Let $\mathcal{A}=(\mathcal{A}^{\bullet},d,c)$ be a curved dg-algebra and $Q=(\mathcal{Q},\mathbb{Q})$ be a quasi-cohesive
module over $\mathcal{A}$. Then there is an object $E$ in $\mathcal{P}_{\mathcal{A}}$ together with a quasi-isomorphism $\phi: E\to Q $  in $\mathcal{C}_{\mathcal{A}}$, under either of the two following conditions:
\begin{enumerate}
\item $Q$ is a quasi-finite quasi-cohesive module (see Definition 3.12 of \cite{block2010duality});
\item $\mathcal{A}$ is flat over $\mathcal{A}^0$ and there is a bounded complex $(E,\mathbb{E}^0)$ of finitely generated projective right $\mathcal{A}^0$-modules and an $\mathcal{A}^0$-linear quasi-isomorphism $\Theta^0:(E,\mathbb{E}^0)\to (\mathcal{Q},\mathbb{Q}^0)$.
\end{enumerate}
\end{thm}
\begin{proof}
See \cite{block2010duality} Theorem 3.13.
\end{proof}

Basically Theorem \ref{thm: quasi-representability of h tilde X} says that, under mild conditions, we can lift quasi-isomorphisms between  underline cochain complexes to quasi-isomorphisms between cohesive modules . We will combine 
Theorem \ref{thm: quasi-representability of h tilde X} and Proposition \ref{prop: glue of smooth perf complex} in the descent problem of cohesive modules.

\begin{prop}\label{prop: there is a cohesive module for the sheafification, general case}
 Let $X$ be a  complex manifold and $\{U_i\}$ be a locally finite open cover, then for every globally bounded twisted  complex $\mathcal{F} \in \text{Tw}^{\text{B}}(X, \mathcal{P}_{\mathcal{A}}, U_i)$, there is an object $\mathcal{E} \in \mathcal{P}_{\mathcal{A}(X)}$  together with a quasi-isomorphism  $\phi:  \mathcal{E}\to  \mathcal{S}(\mathcal{F})$ in $\mathcal{C}_{\mathcal{A}(X)}$.

 Moreover, the corresponding morphism $\eta_{\mathcal{F}}\circ \mathcal{T}(\phi):\mathcal{T}(\mathcal{E})\to \mathcal{F}$ is a homotopy equivalence in $\text{Tw}^{\text{B}}(X, \mathcal{P}_{\mathcal{A}}, U_i)$.
 \end{prop}
 \begin{proof}
 All functors are compatible with the forgetful functor $\For: \mathcal{P}_{\mathcal{A}}\to \mathcal{P}_{\mathcal{A}^0}$. By  Proposition \ref{prop: there is a complex of vector bundles for the sheafification}, for every twisted  complex $\mathcal{F} \in \text{Tw}^{\text{B}}(X, \mathcal{P}_{\mathcal{A}}, U_i)$, there is an object $\mathcal{E}_0 \in \mathcal{P}_{\mathcal{A}^0(X)}$  together with a quasi-isomorphism $\phi_0: \mathcal{E}_0\to \mathcal{S}(\For \mathcal{F})=\For(\mathcal{S}(\mathcal{F}))$.

 It is clear that $\mathcal{A}^{\bullet}(X)$ is flat over $\mathcal{A}^0(X)$. Therefore Condition 2 in Theorem \ref{thm: quasi-representability of h tilde X} is satisfied, hence by Theorem \ref{thm: quasi-representability of h tilde X} there exists an object $\mathcal{E} \in \mathcal{P}_{\mathcal{A}(X)}$  together with a quasi-isomorphism $\phi: \mathcal{E}\to \mathcal{S}(\mathcal{F})$ in $\mathcal{C}_{\mathcal{A}(X)}$ such that $\For(\mathcal{E})=\mathcal{E}_0$ and $\For(\phi)=\phi_0$.

 The second half of the claim comes from Proposition \ref{prop: there is a complex of vector bundles for the sheafification} and the compatibility of the forgetful functor.
 \end{proof}

 \begin{prop}\label{prop: cohesive modules and the sheafification of its twist are homotopy equivalent}
 Let $\mathcal{E}$ be an object in $\mathcal{P}_{\mathcal{A}(X)}$. Apply Proposition \ref{prop: there is a cohesive module for the sheafification, general case} to $\mathcal{T}(\mathcal{E})$, we get $\widetilde{\mathcal{E}}$  in $\mathcal{P}_{\mathcal{A}(X)}$ together with a quasi-isomorphism $\phi: \widetilde{\mathcal{E}}\to \mathcal{S}(\mathcal{T}(\mathcal{E}))$. Then $\mathcal{E}$ and $\widetilde{\mathcal{E}}$ are homotopy equivalent in $\mathcal{P}_{\mathcal{A}(X)}$.
 \end{prop}
 \begin{proof}
 By Proposition \ref{prop: adjunction of twisted functor and sheafification functor}, the unit morphism $\epsilon: \mathcal{E}\to \mathcal{S}(\mathcal{T}(\mathcal{E}))$ is a quasi-isomorphism. Then by Proposition \ref{prop: quasi-isomorphism of quasi-cohesive modules induces quasi-isomorphism for Yoneda embedding}, $\phi: \widetilde{\mathcal{E}}\to \mathcal{S}(\mathcal{T}(\mathcal{E}))$ could be lifted to $\psi: \widetilde{\mathcal{E}}\to \mathcal{E}$ such that $\epsilon\circ \psi=\phi$ up to homotopy. Since both $\epsilon$ and $\phi$ are quasi-isomorphisms, so is $\psi$. Since both $\widetilde{\mathcal{E}}$ and $\mathcal{E}$ are in $\mathcal{P}_{\mathcal{A}(X)}$, by Proposition \ref{prop: homotopy equivalences in PA}, $\psi$ is a homotopy equivalence.
 \end{proof}

We will use the following lemma on dg-categories  in the proof.

\begin{lemma}\label{lemma: adjunctions of big and small dg-categories}[\cite[Lemma 2.3]{ben2013milnor}]
Suppose that $\mathcal{C}$ and $\mathcal{D}$ are dg-categories, which are full dg-subcategories
of dg-categories $\mathcal{C}_{\text{big}}$ and $\mathcal{D}_{\text{big}}$, respectively. Let $F$
be a dg-functor from $\mathcal{C}_{\text{big}}$ to $\mathcal{D}_{\text{big}}$ which carries $\mathcal{C}$ into $\mathcal{D}$. Let $G$ be a dg-functor
from $\mathcal{D}_{\text{big}}$ to $\mathcal{C}_{\text{big}}$ which is right adjoint to $F$. Suppose that $F$ and $G$ satisfy the following conditions:
\begin{enumerate}
\item For each $d\in \mathcal{D}$ we have an object $c_d\in \mathcal{C}$ and a quasi-isomorphism $h_{c_d}\to \tilde{h}_{G(d)}$ in $\Modb \mathcal{C}$, where $h$ and $\tilde{h}$ are Yoneda embeddings;

\item  For each $d\in \mathcal{D}$, let $c_d$ be as above and $\Theta_d\in \mathcal{C}_{\text{big}}(c_d,G(d))$ correspond to the identity of $c_d$ under $h_{c_d}(c_d)\to h_{G(d)}(c_d)$. Then the morphism
$$
\Lambda_d:=\eta_d\circ F(\Theta_d): F(c_d)\to d
$$
is a homotopy equivalence in the dg-category $\mathcal{D}$, where $\eta: F\circ G\to \text{id}_{\mathcal{D}}$ is the counit of the adjunction;

\item For each $c\in \mathcal{C}$, $c$ and $c_{F(c)}$ are homotopy equivalent in $\mathcal{C}$.
\end{enumerate}
Then $F|_{\mathcal{C}}$ is a dg-quasi-equivalence from $\mathcal{C}$ to $\mathcal{D}$.
\end{lemma}
\begin{proof}
See \cite[Lemma 2.3]{ben2013milnor}.
\end{proof}

Now we are ready to prove the main theorem of this paper.

\begin{thm}\label{thm: cohesive modules satisfies descent for open covers}
Let $X$ be a  complex manifold and $\{U_i\}$ be a locally finite open cover of $X$. Let $\mathcal{A}=(\mathcal{A}^{0,\bullet},\bar{\partial},0)$ be the Dolbeault dg-algebra on $X$ and $\mathcal{P}_{\mathcal{A}}$ be the dg-category of cohesive modules. Let $\text{Tw}^{\text{B}}(X, \mathcal{P}_{\mathcal{A}}, U_i)$ be the dg-category of globally bounded twisted complexes on $X$. Then the twisting functor
$$
\mathcal{T}: \mathcal{P}_{\mathcal{A}(X)}\to \text{Tw}^{\text{B}}(X, \mathcal{P}_{\mathcal{A}}, U_i)
$$
is a dg-quasi-equivalence of dg-categories.

In particular, if $\{U_i\}$ be a  finite open cover of $X$, then the twisting functor
$$
\mathcal{T}: \mathcal{P}_{\mathcal{A}(X)}\to \text{Tw}(X, \mathcal{P}_{\mathcal{A}}, U_i)
$$
is a dg-quasi-equivalence of dg-categories.
\end{thm}
\begin{proof}
We want to apply Lemma \ref{lemma: adjunctions of big and small dg-categories} to the adjunction
$$
\mathcal{T}: \mathcal{C}_{\mathcal{A}(X)} \rightleftarrows \text{Tw}^{\text{B}}(X, \mathcal{C}_{\mathcal{A}}, U_i): \mathcal{S}.
$$
In this case $\mathcal{C}=\mathcal{P}_{\mathcal{A}(X)}$, $\mathcal{C}_{\text{big}}=\mathcal{C}_{\mathcal{A}(X)}$, $\mathcal{D}= \text{Tw}(X, \mathcal{P}_{\mathcal{A}}, U_i)$, and $\mathcal{D}_{\text{big}}= \text{Tw}(X, \mathcal{C}_{\mathcal{A}}, U_i)$.

Condition 1 of Lemma \ref{lemma: adjunctions of big and small dg-categories} is obtained by the first assertion of Proposition \ref{prop: there is a cohesive module for the sheafification, general case} and Proposition \ref{prop: quasi-isomorphism of quasi-cohesive modules induces quasi-isomorphism for Yoneda embedding}. Condition 2 is the second assertion of Proposition \ref{prop: there is a cohesive module for the sheafification, general case}, and Condition 3 is given by Proposition \ref{prop: cohesive modules and the sheafification of its twist are homotopy equivalent}. Therefore Lemma \ref{lemma: adjunctions of big and small dg-categories} tells us that $$
\mathcal{T}: \mathcal{P}_{\mathcal{A}(X)}\to \text{Tw}(X, \mathcal{P}_{\mathcal{A}}, U_i)
$$
is a dg-quasi-equivalence of dg-categories.
\end{proof}

\begin{rmk}
It is interesting to generalize the result in Theorem \ref{thm: cohesive modules satisfies descent for open covers} to \emph{hypercovers} as defined in \cite{dugger2004hypercovers}, which is important because it has been shown that the suitable sheaf condition on dg-presheaves is that they satisfies descent for hypercovers instead of merely covers. See \cite[Lemma 6]{holstein2015morita}. However, the author believes that we should leave this problem to a future project.
\end{rmk}

\bibliography{Descentofcohbib}{}
\bibliographystyle{plain}

\end{document}